\documentclass{amsart}
\usepackage{ifthen}
\usepackage{graphicx}
\usepackage{amssymb}
\usepackage{booktabs}
\usepackage{caption}

\newtheorem{proposition}{Proposition}[section]
\newtheorem{theorem}[proposition]{Theorem}

\newtheorem{prop}[proposition]{Proposition}

\newtheorem{conj}[proposition]{Conjecture}

\theoremstyle{definition}
\newtheorem{example}[proposition]{Example}

\theoremstyle{remark}
\newtheorem{remark}[proposition]{Remark}

\newcommand{\set}[1]{{\left\lbrace #1 \right\rbrace}}

\renewcommand{\H}{\mathcal{H}}

\newcommand{\gRec}{{\operatorname{gRec}}}

\renewcommand{\th}{^{\mbox{\footnotesize th}}}

\newcommand{\covered}{\lessdot}

\newcommand{\circlenum}[1]{
\raisebox{-2pt}{
\begin{picture}(10,10)
\put(5,5){\circle{10}}
\put(0.85,2.8){\scriptsize #1}
\end{picture}
}}

\usepackage[usenames]{color}

\newcommand{\margincolor}{Red}

\newcounter{margincounter}
\newcommand{\marginnum}{\textcolor{\margincolor}{\begin{picture}(0,0)\put(5,3){\circle{13}}\end{picture}\arabic{margincounter}}}

\newcommand{\margin}[1]
{\marginnum\marginpar{\textcolor{\margincolor}{\arabic{margincounter}} \tiny #1}\addtocounter{margincounter}{1}}

\newcommand{\proofread}
{
\ifthenelse{\isundefined{\margin}}
{
\special{!userdict begin /bop-hook{1.2 1.2 scale -51 -60 translate}def end}
}
{
\special{!userdict begin /bop-hook{1.2 1.2 scale -91 -60 translate}def end}
\setboolean{@mparswitch}{false} 
\addtolength{\marginparwidth}{-3mm}
}
}

\author{Nathan Reading}
\address{Department of Mathematics, North Carolina State University, Raleigh, NC, USA}


\thanks{\textbf{Acknowledgments:}
Thanks to Shirley Law for helpful conversations.
Nathan Reading was partially supported by NSA grant H98230-09-1-0056.}

\title{Generic rectangulations}

\begin{document}

\begin{abstract}
A rectangulation is a tiling of a rectangle by a finite number of rectangles.
The rectangulation is called generic if no four of its rectangles share a single corner.
We initiate the enumeration of generic rectangulations up to combinatorial equivalence by establishing an explicit bijection between generic rectangulations and a set of permutations defined by a pattern-avoidance condition analogous to the definition of the twisted Baxter permutations.
\end{abstract}


\maketitle

\setcounter{tocdepth}{2}
\tableofcontents

\section{Introduction}

The main characters in this paper are tilings of a rectangle by finitely many rectangles.
A \emph{cross} in such a tiling is a point which is a corner of four distinct tiles.
Fixing a rectangle $S$ and considering the space of all tilings of $S$ by $n$ rectangles, with a uniform probability measure, the set of tilings having one or more crosses has measure zero.
Thus we call a tiling \emph{generic} if it has no crosses.

We consider generic tilings up to the natural combinatorial equivalence relation which we now describe.
We orient $S$ so that its edges are vertical and horizontal.
A rectangle $U$ in a tiling $R$ is \emph{below} a rectangle $V$ if the top edge of $U$ intersects the bottom edge of $V$ (necessarily in a line segment rather than in a point).
Similarly, $U$ is \emph{left of} $V$ if the right edge of $U$ intersects the left edge of $V$.
A tiling $R$ of a rectangle $S$ is combinatorially equivalent to a tiling $R'$ of a rectangle $S'$ if there is a bijection from the rectangles of $R$ to the rectangles of $R'$ that exactly preserves the relations ``below'' and ``left of.'' 
A \emph{generic rectangulation} is the equivalence class of a generic tiling.
We will often blur the distinction between generic rectangulations (i.e.\ equivalence classes) and equivalence class representatives, in particular specifying an equivalence class by describing a specific tiling.

Our main result is a bijection between generic rectangulations with $n$ rectangles and a class of permutations in $S_n$ that we call \emph{$2$-clumped permutations}.
These are the permutations that avoid the patterns $3$-$51$-$2$-$4$, $3$-$51$-$4$-$2$, $2$-$4$-$51$-$3$, and $4$-$2$-$51$-$3$, in the notation of Babson and Steingr\'{\i}msson~\cite{BaSt}, which is explained in Section~\ref{clumped sec}.
The author's counts of generic rectangulations, for small $n$, are shown in Table~\ref{comput table}.

\begin{figure}
\begin{tabular}{rrr}
\toprule
$n$&&Generic rectangulations\\
\midrule
1&&1\\
2&&2\\
3&&6\\
4&&24\\
5&&116\\
6&&642\\
7&&3,938\\
8&&26,194\\
9&&186,042\\
10&&1,395,008\\
11&&10,948,768\\
12&&89,346,128\\
13&&754,062,288\\
14&&6,553,942,722\\
\bottomrule
\end{tabular}
\qquad\quad
\begin{tabular}{rrr}
\toprule
$n$&&Generic rectangulations\\
\midrule
15&&58,457,558,394\\
16&&533,530,004,810\\
17&&4,970,471,875,914\\
18&&47,169,234,466,788\\
19&&455,170,730,152,340\\
20&&4,459,456,443,328,824\\
21&&44,300,299,824,885,392\\
22&&445,703,524,836,260,400\\
23&&4,536,891,586,511,660,256\\
24&&46,682,404,846,719,083,048\\
25&&485,158,560,873,624,409,904\\
26&&5,089,092,437,784,870,584,576\\
27&&53,845,049,871,942,333,501,408\\
\\
\bottomrule
\end{tabular}
\captionsetup{name=Table}
\caption{The number of generic rectangulations with $n$ rectangles}
\label{comput table}
\end{figure}

We define $k$-clumped permutations in Section~\ref{clumped sec}.
For now, to place the $2$-clumped permutations in context, we note that the $1$-clumped permutations are the \emph{twisted Baxter permutations}, which are in bijection with the better-known \emph{Baxter permutations}.
Baxter permutations are also relevant to the combinatorics of rectangulations.
Indeed, Baxter permutations are in bijection \cite{ABP2,YCCG} with the \emph{mosaic floorplans} considered in the VLSI (Very Large Scale Integration) circuit design literature~\cite{HHCGDCG}.
Mosaic floorplans are certain equivalence classes of generic rectangulations.
(A similar result linking equivalence classes of generic rectangulations to pattern-avoiding permutations is given in \cite{ABBMP}.)
In light of results of \cite{ABP}, the bijection from Baxter permutations to mosaic floorplans can be rephrased as a bijection to a subclass of the generic rectangulations that we call \emph{diagonal rectangulations}, which figure prominently in this paper.

The symbol $G_n$ will denote the set of $2$-clumped permutations.
Let $\gRec_n$ be the set of generic rectangulations with $n$ rectangles.
The bijection from $G_n$ to $\gRec_n$ is defined as the restriction of a map $\gamma:S_n\to\gRec_n$.
We show that $\gamma$ is surjective and that its fibers are the congruence classes of a lattice congruence on the weak order on~$S_n$.
We do not prove directly that the fibers of $\gamma$ define a congruence.
Instead, we recognize the fibers as the classes of a congruence arising as one case of a construction from~\cite{con_app}, where lattice congruences on the weak order are used to construct sub Hopf algebras of the Malvenuto-Reutenauer Hopf algebra of permutations.
The results of~\cite{con_app} show that the $2$-clumped permutations are a set of congruence class representatives.
Thus the restriction of $\gamma$ is a bijection from $G_n$ to $\gRec_n$.

\subsection*{Note added in proof}
After this paper was accepted, the author became aware of a substantial literature studying generic rectangulations under the name \emph{rectangular drawings}.
This literature includes some results on asymptotic enumeration as well as computations of the exact cardinality of $\gRec_n$ for many values of $n$.
See, for example, \cite{ANY,FT,ITF}.
In particular, the main result of this paper answers an open question posed in \cite[Section~5]{ANY}.

\section{Clumped permutations}\label{clumped sec}
In this section, we define $k$-clumped permutations.
We begin with a review of generalized pattern avoidance in the sense of Babson and Steingr\'{\i}msson~\cite{BaSt}.
Let $y=y_1\cdots y_k\in S_k$, and let $\tilde{y}$ be a word created by inserting a dash between some letters of $y_1\cdots y_k$, with at most one dash between each adjacent pair.
A subsequence $x_{i_1}\cdots x_{i_k}$ of $x_1\cdots x_n$ is an \emph{occurrence} of the pattern $\tilde{y}$ in a permutation $x\in S_n$ if the following two conditions are satisfied:
First, for all $j,l\in[k]$ with $j<l$, the inequality $x_{i_j}<x_{i_l}$ holds if and only if $y_i<y_l$ holds.
Second, if $y_j$ and $y_{j+1}$ are not separated by a dash in $\tilde{y}$, then $i_{j}=i_{j+1}-1$.
That is, the dashes indicate which elements of the subsequence are not required to be adjacent in $x$.
For example, the subsequence $4512$ of $45312\in S_5$ is an occurrence of the pattern $3$-$4$-$1$-$2$, or an occurrence of the pattern $34$-$12$, but not an occurrence of the pattern $3$-$41$-$2$.
If there is no occurrence of the pattern $\tilde{y}$ in $x$, then we say that $x$ \emph{avoids} $\tilde{y}$.

To define $k$-clumped permutations, we first consider the \emph{twisted Baxter permutations}, defined in~\cite{con_app} and shown in unpublished notes by West~\cite{West pers} to be in bijection with Baxter permutations.
A published proof can be found in \cite{rectangle} or \cite{Giraudo}.
The twisted Baxter permutations are the permutations that avoid the patterns $2$-$41$-$3$ and $3$-$41$-$2$.
This pattern-avoidance condition on a permutation $x=x_1\cdots x_n$ can be rephrased as follows:
For every descent $x_i>x_{i+1}$, the values strictly between $x_{i+1}$ and $x_i$ are either all to the left of $x_i$ or all to the right of $x_{i+1}$.
(The Baxter permutations are defined by a similar condition:  They are the permutations avoiding $3$-$14$-$2$ and $2$-$41$-$3$.)

In any permutation $x$, we define a \emph{clump} associated to a descent $x_i>x_{i+1}$ to be a nonempty maximal sequence of consecutive values strictly between $x_i$ and $x_{i+1}$, all of which are on the same side of the entries $x_ix_{i+1}$.
No requirement is made on the positions, relative to each other, of the values in the clump.
For example, in the permutation $269153847\in S_9$, there are four clumps associated to the descent $9>1$, namely $2$, $345$, $6$, and $78$.

The pattern avoidance condition defining twisted Baxter permutations is that each descent $x_i>x_{i+1}$ has at most one clump, so we refer to twisted Baxter permutations as $1$-clumped permutations.
More generally, a \emph{$k$-clumped permutation} is a permutation $x$ such that each descent $x_i>x_{i+1}$ has at most $k$ clumps.
One can easily rephrase the definition of $k$-clumped permutations in terms of generalized patterns avoidance (avoiding $2\left(\frac{k}{2}\right)!\left(\frac{k}{2}+1\right)!$ generalized patterns if $k$ is even or $2\left(\frac{k+1}{2}\right)!\left(\frac{k+1}{2}\right)!$ generalized patterns if $k$ is odd).
For example, the $2$-clumped permutations, which play the central role in this paper, are the permutations avoiding $3$-$51$-$2$-$4$, $3$-$51$-$4$-$2$, $2$-$4$-$51$-$3$, and $4$-$2$-$51$-$3$.
By convention, the only $(-1)$-clumped permutation is the identity.
The $0$-clumped permutations are the permutations such that if $x_i>x_{i+1}$ then $x_i-1=x_{i+1}$.
Equivalently, they are the permutations avoiding $31$-$2$ and $2$-$31$.
These permutations in $S_n$ are in bijection with subsets of $\set{1,2,\cdots n-1}$.
The $3$-clumped permutations appear not to have been considered before.
For $n$ from $1$ to $9$, the numbers of $3$-clumped permutations are $1$, $2$, $6$, $24$, $120$, $712$, $4804$, $35676$ and $284816$.

The \emph{weak order} on $S_k$ is a lattice whose cover relations are $x\covered y$ with $x=x_1\cdots x_k$ and $y=y_1\cdots y_k$ such that $x_i=y_{i+1}<y_i=x_{i+1}$ for some $i\in[k-1]$, with $x_j=y_j$ for $j\not\in\set{i,i+1}$.
A \emph{join-irreducible} permutation is a permutation $x\in S_k$ with exactly one descent, meaning that, for some index~$i\in[k-1]$, we have $x_i>x_{i+1}$ but $x_{j}<x_{j+1}$ for every $j\in[k-1]$ with $j\neq i$.
(Such a permutation is join-irreducible in the weak order in the usual lattice-theoretic sense.)

We now review a construction from \cite[Section~9]{con_app}.
A join-irreducible element $x\in S_k$ is called \emph{untranslated} if its unique descent $x_i>x_{i+1}$ has $x_i=k$ and $x_{i+1}=1$.
In this case, a \emph{scramble} of $x$ is any permutation $y$ such that $y_i=k$, $y_{i+1}=1$ and every entry $j$ with $1<j<k$ occurs to the left of position $i$ in $x$ if and only if it occurs to the left of position $i$ in $y$.
Let $y$ be a scramble of $x$ and let~$\tilde{y}$ be obtained from $y$ by inserting a dash between each pair of consecutive entries except between $k$ and $1$.
We say that the scramble $y$ of $x$ occurs \emph{with adjacent cliff} if the pattern $\tilde{y}$ occurs.

Let $C$ be any collection of untranslated join-irreducible elements in $S_k$, with $k$ varying, so that, for example, $C$ may be $\set{312,2413}$.
The following is essentially \cite[Theorem~9.3]{con_app}.
\begin{theorem}\label{H fam combin}
For each $n$, there exists a unique congruence $\H(C)_n$ on the weak order on $S_n$ with the following properties:
\begin{enumerate}
\item[(i) ] A permutation $z$ is the minimal element in its $\H(C)_n$-class if and only if, for every $x\in C$ and all scrambles $y$ of $x$, the permutation $z$ avoids occurrences of $y$ with adjacent cliff.
\item[(ii) ] Suppose $w\covered z$ in the weak order, and let $z_i$ and $z_{i+1}$ be the adjacent entries of $z$ that are swapped to convert $z$ to $w$, with $z_i>z_{i+1}$.
Then $w\equiv z$ modulo $\H(C)_n$ if and only there exists $x\in C$, a scramble $y\in S_k$ of $x$, and an occurrence of $\tilde{y}$ in $z$ such that the entry of $z$ corresponding to the entry $k$ in $\tilde{y}$ is $z_i$ and the entry of $z$ corresponding to $1$ in $y$ is $z_{i+1}$.
\end{enumerate}
\end{theorem}
In \cite{con_app}, the congruence $\H(C)_n$ is constructed for the purpose of building combinatorial Hopf algebras.
Here, we can take Theorem~\ref{H fam combin} as the definition of $\H(C)_n$.
Property (i) in Theorem~\ref{H fam combin} is a direct restatement of \cite[Theorem~9.3]{con_app}, while property (ii) is the key point in the proof of \cite[Theorem~9.3]{con_app}. 
It is easy and well-known that in a congruence on a finite lattice, each congruence class is an interval.
Thus a congruence is uniquely determined by the set of cover relations $w\covered z$ such that $w\equiv z$.
Furthermore, the minimal permutations described in Property (i) are a system of congruence class representatives.

Let $\Gamma$ be the congruence $\H(\set{35124,24513})_n$ on $S_n$.
Theorem~\ref{H fam combin} specializes to the following:
\begin{prop}\label{Theta g props}

\noindent
\begin{enumerate}
\item \label{Theta g bottoms}
A permutation is the minimal element in its $\Gamma$-class if and only if it is a $2$-clumped permutation.
\item \label{Theta g moves}
Suppose $x\covered y$ in the weak order, and let $e$ and $a$ be the adjacent entries that are swapped to convert $y$ to $x$, with $a<e$.
Then $x\equiv y$ modulo $\Gamma$ if and only if there are entries $b$, $c$, and $d$ in $y$ with $a<b<c<d<e$ such that $b$ and $d$ are on the same side of $ea$, while $c$ is on the other side of $ea$.
\end{enumerate}
\end{prop}

More generally, for each $k\ge -1$, there is a congruence described by Theorem~\ref{H fam combin} such that the minimal elements of congruence classes are exactly the $k$-clumped permutations.

\section{The map from permutations to generic rectangulations}
In this section, we define a map $\gamma$ from $S_n$ to $\gRec_n$.
We will see, in Section~\ref{main sec}, that $\gamma$ restricts to a bijection from the set of $2$-clumped permutations to $\gRec_n$.
The key point in the proof that $\gamma$ restricts to a bijection will be the fact that its fibers are the congruence classes of the congruence $\Gamma$ defined at the end of Section~\ref{clumped sec}.

To define the map $\gamma$, we first consider a smaller class of rectangulations which we call \emph{diagonal rectangulations} and a map from permutations to diagonal rectangulations.
The \emph{diagonal} of the underlying rectangle $S$ is the line segment connecting the top-left corner of $S$ to the bottom-right corner of $S$.
Recall that each rectangulation is a combinatorial equivalence class.
A rectangulation is a diagonal rectangulation if it has a representative in which each rectangle's interior intersects the diagonal.
A diagonal rectangulation is in particular a generic rectangulation, because if any four rectangles have a common vertex, it is impossible for all of their interiors to intersect the diagonal.
Diagonal rectangulations have been considered under other names, for example in~\cite{ABP,DGStack,FFNO}.

We now review, from~\cite{rectangle}, the definition of a map $\rho$ from permutations to diagonal rectangulations.
Maps closely related to $\rho$ have appeared prior to~\cite{rectangle}, for example in~\cite{ABP,FFNO}.
To define $\rho$, first draw $n+1$ distinct \emph{diagonal points} on the diagonal of $S$, with one of the points being the top-left corner of $S$ and another being the bottom-right corner of $S$.
Number the spaces between the diagonal points as $1,2,\ldots,n$, from top-left to bottom-right.
Given $x\in S_n$, read the sequence $x_1\cdots x_n$ from left to right and draw a rectangle for each entry according to the following recursive procedure:

Let $T$ be the union of the left and bottom edges of $S$ with the rectangles drawn in the first $i-1$ steps of the construction.
It will be apparent by induction that $T$ is left- and bottom-justified.
To draw the $i\th$ rectangle, consider the label $x_i$ on the diagonal.
If the diagonal point $p$ immediately above/left of the label $x_i$ is not in $T$, then the top-left corner of the new rectangle is the rightmost point of $T$ that is directly left of $p$.
If $p$ is in $T$ (necessarily on the boundary of~$T$), then the top-left corner of the new rectangle is the highest point of $T$ directly above $p$.
If the diagonal point $p'$ immediately below/right of the label $x_i$ is not in~$T$, then the bottom-right corner of the new rectangle is the highest point of $T$ that is directly below $p'$.
If $p'$ is in $T$ then the bottom-right corner of the new rectangle is the rightmost point of $T$ that is directly to the right of $p'$.

\begin{example}\label{rho ex}
Figure~\ref{rho fig} illustrates the map $\rho$.
In each step, the new rectangle is shown in red (the darker gray when not viewed in color), and the set $T$ consists of the white rectangles together with the left and bottom edges of $S$.
The part of $S$ not covered by rectangles is shaded in light gray.
\end{example}
\setcounter{figure}{0}
\begin{figure}
\begin{tabular}{ccccc}
\scalebox{.85}{\includegraphics{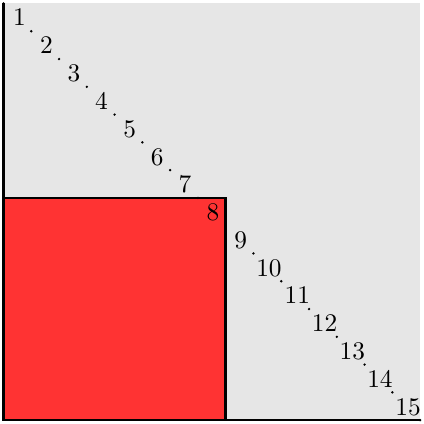}}&&
\scalebox{.85}{\includegraphics{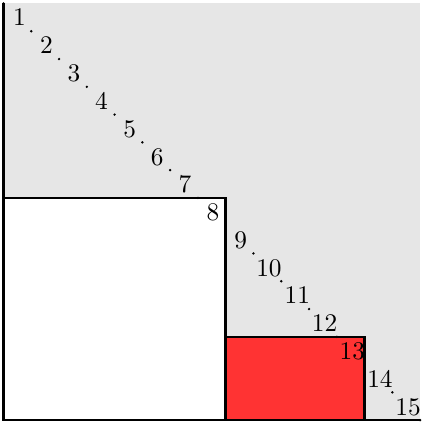}}&&
\scalebox{.85}{\includegraphics{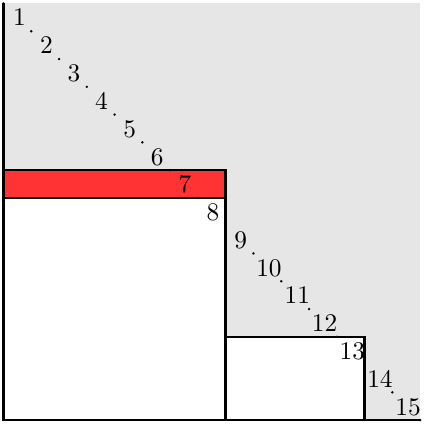}}\\[6.5 pt]
\scalebox{.85}{\includegraphics{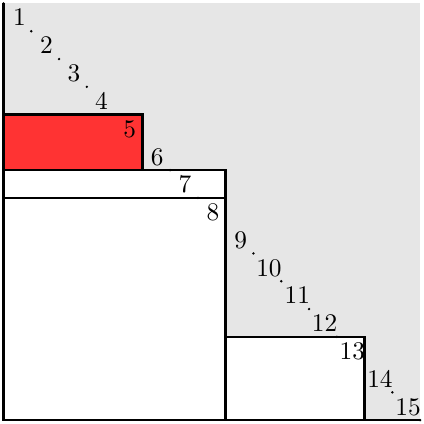}}&&
\scalebox{.85}{\includegraphics{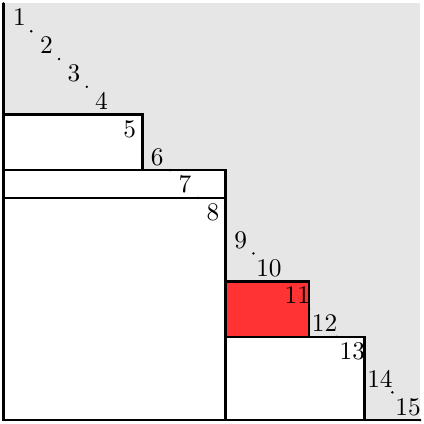}}&&
\scalebox{.85}{\includegraphics{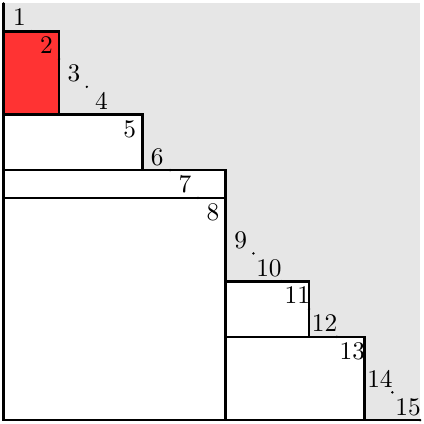}}\\[6.5 pt]
\scalebox{.85}{\includegraphics{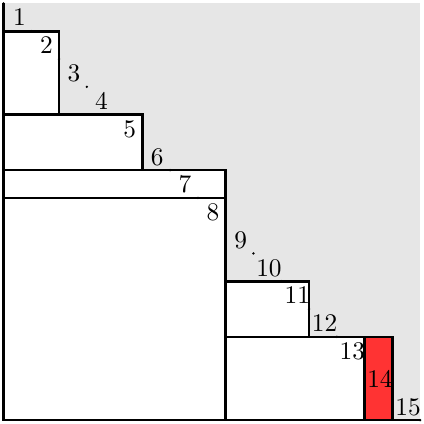}}&&
\scalebox{.85}{\includegraphics{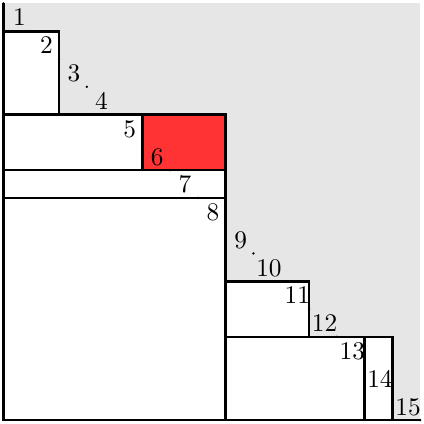}}&&
\scalebox{.85}{\includegraphics{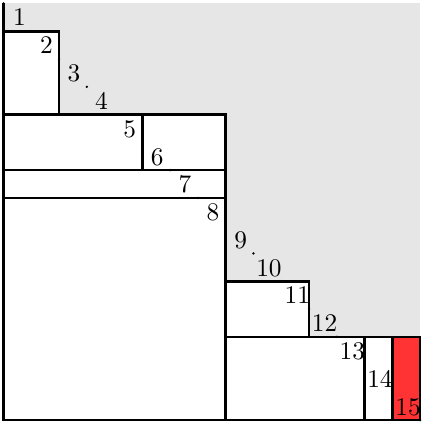}}\\[6.5 pt]
\scalebox{.85}{\includegraphics{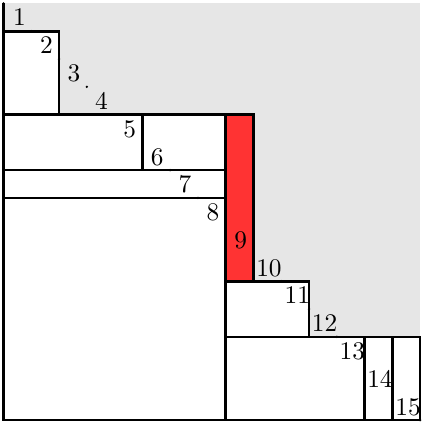}}&&
\scalebox{.85}{\includegraphics{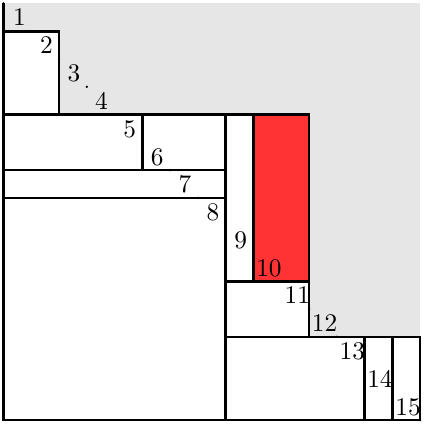}}&&
\scalebox{.85}{\includegraphics{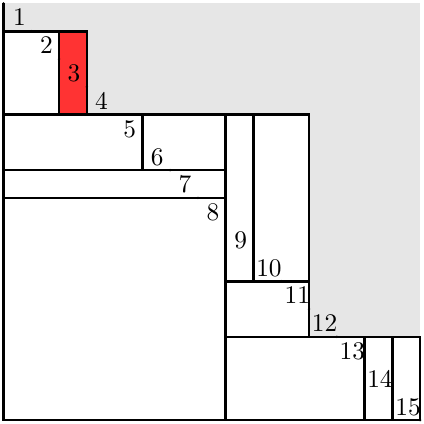}}\\[6.5 pt]
\scalebox{.85}{\includegraphics{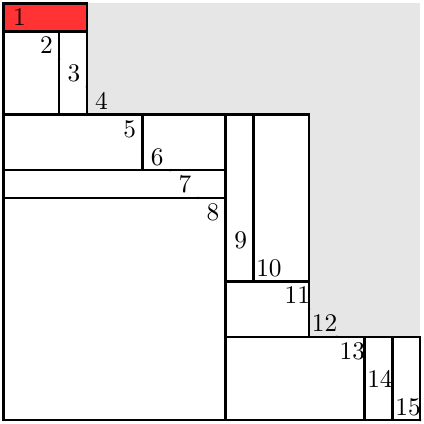}}&&
\scalebox{.85}{\includegraphics{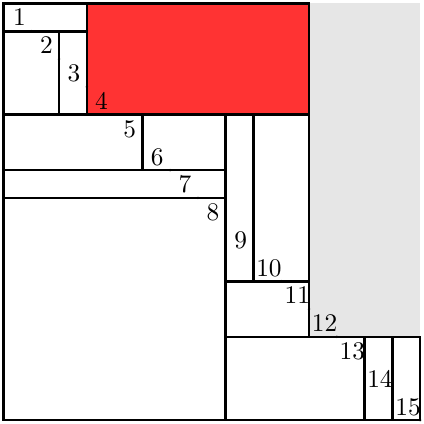}}&&
\scalebox{.85}{\includegraphics{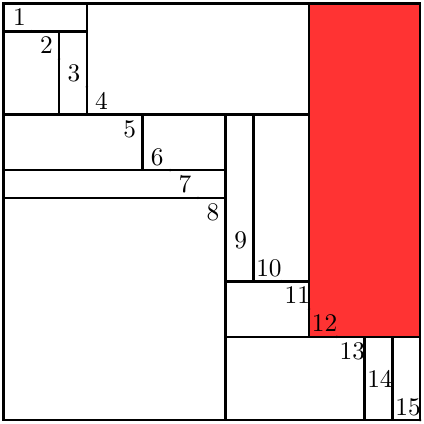}}
\end{tabular}
\caption{Steps in the construction of $\rho(8\;\!\!\!\protect\circlenum{13}\;\!\!\!75\;\!\!\!\protect\circlenum{11}\;\!\!\!2\;\!\!\!\protect\circlenum{14}\;\!\!\!6\;\!\!\!\protect\circlenum{15}\;\!\!\!9\;\!\!\!\protect\circlenum{10}\;\!\!\!314\;\!\!\!\protect\circlenum{12}\:\!\!\!)$}
\label{rho fig}
\end{figure}

Given a diagonal rectangulation $R$, we number the rectangles in $R$ according to the position of their intersections with the diagonal, starting with rectangle 1, which contains the top-left corner of $S$ and ending at rectangle $n$, which contains the bottom-right corner of $S$.
Thus, for example, in constructing the rectangulation $\rho(x)$, we first construct the rectangle numbered $x_1$, then the rectangle numbered $x_2$, etc.
Say a permutation $x=x_1\cdots x_n$ is \emph{compatible with} a diagonal rectangulation $R$ if and only if, for every $i\in [n]$, the left and bottom edges of the rectangle numbered $x_i$ are contained in the union of the left and bottom edges of $S$ with the rectangles numbered $x_1,\ldots,x_{i-1}$.
Equivalently, $x$ is compatible with $R$ if, for every $i\in [n]$, the union of the rectangles numbered $x_1,\ldots,x_i$ is left- and bottom-justified. 
The following fact is established in the proof \cite[Proposition~6.2]{rectangle}, which asserts that $\rho$ is surjective.

\begin{prop}\label{rho fibers}
Given a diagonal rectangulation $R$, the fiber $\rho^{-1}(R)$ is the set of permutations in $S_n$ that are compatible with~$R$.
\end{prop}

The following proposition, which is the concatenation of \cite[Proposition~4.5]{rectangle} and \cite[Theorem~6.3]{rectangle},
shows in particular that the fibers of $\rho$ constitute a congruence of the kind described in Theorem~\ref{H fam combin}.
\begin{prop}\label{tBax moves}
Suppose $x\covered y$ in the weak order, and let $d$ and $a$ be the adjacent entries that are swapped to convert $y$ to $x$, with $a<d$.
Then $\rho(x)=\rho(y)$ if and only if there are entries $b$ and $c$, with $a<b<c<d$, such that $b$ and $c$ are on opposite sides of $da$ in $y$.
\end{prop}

In some of the literature on floorplanning for integrated circuits, generic rectangulations are referred to as \emph{mosaic floorplans}, but in that literature, the term mosaic floorplan always implies a coarser equivalence relation than the combinatorial equivalence used to define rectangulations as equivalence classes.
Specifically, two generic rectangulations are equivalent as mosaic floorplans if and only if they are related by a sequence of what we call \emph{wall slides}.
A \emph{wall} in a rectangulation~$R$ is a line segment in the underlying rectangle $S$, not contained in an edge of $S$, that is maximal with respect to the property of not intersecting the interior of any rectangle of $R$.
A wall slide along a wall $W$ is the operation taking two walls of~$R$ that end in $W$, from opposite sides, and sliding them past each other, without changing any of the other incidences in $R$.
Wall slides come in two orientations, as illustrated in Figure~\ref{slide fig}.
\begin{figure}
\begin{tabular}{ccc}
\includegraphics{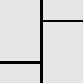}&\raisebox{9.5 pt}{$\longleftrightarrow$}&\includegraphics{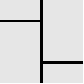}\\[10 pt]
\includegraphics{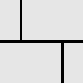}&\raisebox{9.5 pt}{$\longleftrightarrow$}&\includegraphics{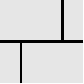}
\end{tabular}
\caption{Wall slides}
\label{slide fig}
\end{figure}
The following is a very special case of \cite[Theorem~4]{ABP}.
\begin{prop}\label{unique diagonal mosaic}
Given a generic rectangulation $R$, there exists a unique diagonal rectangulation $R'$ such that $R$ and $R'$ are equivalent as mosaic floorplans. 
\end{prop}
To see Proposition~\ref{unique diagonal mosaic} as a special case of \cite[Theorem~4]{ABP}, we need the definition of a diagonal rectangulation given in \cite[Section~5]{rectangle}:
Let $X$ be a set of $n-1$ distinct points on the diagonal of $S$, none of which is the top-left corner or bottom-right corner of $S$.  
Then a diagonal rectangulation of $(S,X)$ is a generic rectangulation such that every wall contains a point of $X$ and such that every point of $X$ lies on a wall.  
By \cite[Proposition~5.2]{rectangle}, this definition is equivalent to the earlier definition.

Suppose $R$ is a generic rectangulation and let $R'$ be the diagonal rectangulation that is equivalent to $R$ as a mosaic floorplan.
As before, number the rectangles in $R'$ according to the position of their intersections with the diagonal, $1$ to $n$ from top-left to bottom-right.
Letting this numbering propagate along wall slides in the obvious way, we obtain a numbering of the rectangles of $R$.
For each vertical wall $W$ of $R$, we produce a permutation $\sigma_W$ of a subset of $[n]$ as follows:
Moving from the bottom endpoint of $W$ to the top endpoint of $W$, when we come to a wall $W'$ that is incident to $W$ on the left, we record the number of the rectangle that has its right edge in $W$ and its \emph{bottom} edge in $W'$.
When we come to a wall $W'$ that is incident to $W$ on the right, we record the number of the rectangle that has its left edge in $W$ and its \emph{top} edge in $W'$.
The resulting partial permutation $\sigma_W$ is called the \emph{wall shuffle} associated to $W$, because it is obtained by shuffling two sequences: the decreasing sequence of numbers of rectangles whose right edge is contained in~$W$ (excluding the bottom such rectangle) from bottom to top and the decreasing sequence of numbers of rectangles whose left edge is contained in $W$ (excluding the top such rectangle) from bottom to top.

For each horizontal wall $W$, we construct the wall shuffle associated to $W$ in a similar manner.
Moving from the left endpoint of $W$ to the right endpoint of $W$, when we come to a wall $W'$ that is incident to $W$ on the top, we record the number of the rectangle that has its bottom edge in $W$ and its \emph{right} edge in $W'$.
When we come to a wall $W'$ that is incident to $W$ on the bottom, we record the number of the rectangle that has its top edge in $W$ and its \emph{left} edge in $W'$.
The partial permutation $\sigma_W$, in this case, is obtained by shuffling two increasing sequences: the sequence of numbers of rectangles whose bottom edge is contained in $W$ (excluding the rightmost such rectangle) from left to right and the sequence of numbers of rectangles whose top edge is contained in $W$ (excluding the leftmost such rectangle) from left to right.

\begin{example}\label{wall shuf ex}
Figure~\ref{rec numbered fig} shows a generic rectangulation $R$ whose associated diagonal rectangulation $R'$ is the rectangulation from Figure~\ref{rho fig}.
The numbering of rectangles is inherited from $R'$.
\begin{figure}
\centerline{\includegraphics{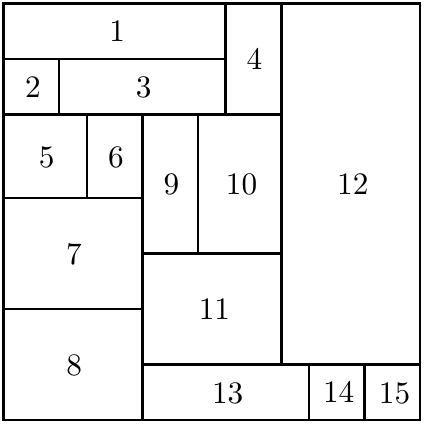}}
\caption{A generic rectangulation}
\label{rec numbered fig}
\end{figure}
Tables~\ref{sh tab 1} and~\ref{sh tab 2} show the wall shuffles associated to~$R$.
\end{example}

\setcounter{figure}{1}
\begin{figure}
\begin{tabular}{ccc}
\toprule
Rectangles left of wall &Rectangles right of wall &Wall shuffle\\
\midrule
2&3&\textit{empty}\\
5&6&\textit{empty}\\
8, 7, 6&13, 11, 9&\circlenum{13}\;\!\!\!7\!\!\!\circlenum{11}\;\!\!\!6\\
9&10&\textit{empty}\\
3, 1&4&1\\
11, 10, 4&12&\circlenum{10}\;\!\!\!4\\
13&14&\textit{empty}\\
14&15&\textit{empty}\\
\bottomrule
\end{tabular}
\captionsetup{name=Table}
\caption{Wall shuffles in vertical walls of the rectangulation of Figure~\ref{rec numbered fig}}
\label{sh tab 1}
\bigskip
\bigskip
\begin{tabular}{ccc}
\toprule
Rectangles above wall &Rectangles below wall &Wall shuffle\\
\midrule
1&2,3&3\\
2, 3, 4&5, 6, 9, 10&269\!\!\!\circlenum{10}\;\!\!\!3\\
5, 6&7&5\\
9. 10&11&9\\
7&8&\textit{empty}\\
11, 12&13, 14, 15&\circlenum{11}\!\!\!\!\!\circlenum{14}\!\!\!\!\!\circlenum{15}\\
\bottomrule
\end{tabular}
\captionsetup{name=Table,margin=0pt}
\caption{Wall shuffles in horizontal walls of the rectangulation of Figure~\ref{rec numbered fig}}
\label{sh tab 2}
\end{figure}

Specifying a generic representation $R$ is equivalent to specifying the associated diagonal rectangulation $R'$ along with the wall shuffles for each wall.
For some walls, there may be only one shuffle possible, and this unique shuffle may be empty.
The shuffles may be chosen arbitrarily (among shuffles of the appropriate rectangle numbers) and independently for each wall, and each sequence of choices of $R'$ and the wall shuffles yields a different generic rectangulation.

When a wall slide is performed along a wall $W$, the move alters $\sigma_W$ by swapping two adjacent entries which number rectangles on opposite sides of $W$.
Since a wall slide only changes the combinatorics locally, performing a wall slide along $W$ does not alter the wall shuffle for any other wall.

We now define the map $\gamma:S_n\to\gRec_n$.
Let $x=x_1x_2\cdots x_n\in S_n$ and construct $R'=\rho(x)$.  
Let $W$ be a vertical wall in $R'$ and consider the rectangles in $R'$ having their right edges contained in $W$.
By construction, the numbers of these rectangles form a decreasing subsequence of $x_1x_2\cdots x_n$.
Similarly, the numbers of the rectangles in $R'$ having their left edges contained in $W$ are a decreasing subsequence of $x_1x_2\cdots x_n$.
Thus we can specify a wall shuffle $\sigma_W$ by taking the subsequence of $x_1x_2\cdots x_n$ consisting of the appropriate rectangle numbers.
For a horizontal wall $W$, the numbers of the rectangles having their top edges contained in $W$ form an increasing subsequence of $x_1x_2\cdots x_n$ and the numbers of the rectangles having their top edges contained in $W$ form an increasing subsequence of $x_1x_2\cdots x_n$, so, in this case as well, we can specify a wall shuffle for $W$ by taking an appropriate subsequence of $x_1x_2\cdots x_n$.
The diagonal rectangulation $R'$ together with all of these wall shuffles define the generic rectangulation $\gamma(x)$.

\begin{example}\label{gamma ex}
This is a continuation of Examples~\ref{rho ex} and~\ref{wall shuf ex}.
Figure~\ref{rho fig} shows the construction of $\rho(x)$ for $x=8\;\!\!\!\circlenum{13}\;\!\!\!75\;\!\!\!\circlenum{11}\;\!\!\!2\;\!\!\!\circlenum{14}\;\!\!\!6\;\!\!\!\circlenum{15}\;\!\!\!9\;\!\!\!\circlenum{10}\;\!\!\!314\;\!\!\!\circlenum{12}\:\!\!\!$.
To construct $\gamma(x)$, we look at each wall of $\rho(x)$.
For example, $\rho(x)$ has a horizontal wall $W$ with rectangles 2, 3, and 4 above $W$ and rectangles 5, 6, 9, and 10 below $W$.
The restriction of $x$ to the set $\set{2,3,6,9,10}$ is $269\;\!\!\!\circlenum{10}\;\!\!\!3$.
Thus $\gamma(x)$ is a rectangulation that is mosaic equivalent to $\rho(x)$ and that has a wall shuffle $269\;\!\!\!\circlenum{10}\;\!\!\!3$.
Considering similarly the other five horizontal walls of $\rho(x)$ and the eight vertical walls of $\rho(x)$, we see that $\gamma(x)$ is the rectangulation shown in Figure~\ref{rec numbered fig}.
(Cf.\ Tables~\ref{sh tab 1} and~\ref{sh tab 2}.)
\end{example}

\begin{example}\label{gamma S4 ex}
Figure~\ref{gamma S4 fig} shows the map $\gamma$ applied to every permutation in $S_4$.
The permutations in $S_4$ are shown in the weak order, and the 24 rectangulations in $\gRec_4$ are shown in the corresponding order.
As a byproduct of the results of Section~\ref{main sec}, the map $\gamma:S_n\to\gRec_n$ induces a lattice structure on $\gRec_n$ such that $\gamma$ is a surjective lattice homomorphism.
\end{example}
\begin{figure}
\scalebox{1}{\includegraphics{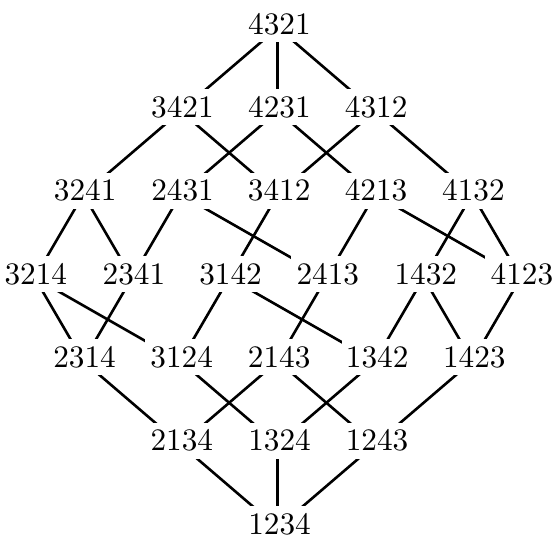}}\quad\scalebox{1}{\includegraphics{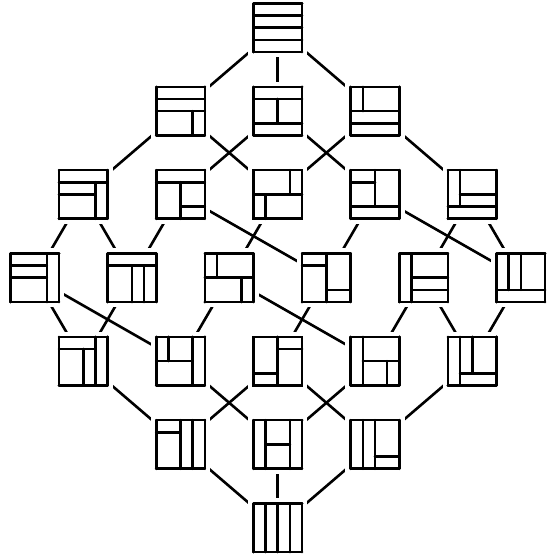}}
\caption{$\gamma:S_4\to\gRec_4$}
\label{gamma S4 fig}
\end{figure}

\section{Main theorem}\label{main sec}
In this section, we prove our main theorem.

\begin{theorem}\label{main}
The restriction of $\gamma$ is a bijection from the set of $2$-clumped permutations in $S_n$ to the set of generic rectangulations with $n$ rectangles.
\end{theorem}

The proof of Theorem~\ref{main} is accomplished by proving three propositions.

\begin{prop}\label{gamma surj}
The map $\gamma:S_n\to \gRec_n$ is surjective.
\end{prop}
\begin{proof}
Let $R'$ be any diagonal rectangulation and choose an arbitrary wall shuffle for each wall of $R'$.
We need to show that there exists $x=x_1x_2\cdots x_n\in S_n$ such that $\rho(x)=R'$ and such that each chosen wall shuffle is a subsequence of $x_1x_2\cdots x_n$.
That is, we need to show that the rectangles of $R'$ can be ordered consistent with the requirements of Proposition~\ref{rho fibers} and with the wall shuffles.

Suppose, for $1\le i\le n$, that we have chosen $i-1$ rectangles in an order consistent with the requirements of Proposition~\ref{rho fibers} and with the wall shuffles.
We will show that we can choose a rectangle in step $i$ that also satisfies the requirements.
Since we have chosen consistent with Proposition~\ref{rho fibers}, the union $T$ of the $i-1$ rectangles chosen with the left and bottom edges of $S$, is a left- and bottom-justified set.
To satisfy the requirement of Proposition~\ref{rho fibers} in step $i$, we must chose a rectangle whose bottom and left edges are contained in $T$.
To show that we can choose such a rectangle consistent with the wall shuffles, we extend an argument from the proof of \cite[Proposition~6.2]{rectangle}.

The top-right boundary of $T$ is a polygonal path from the top-left corner of $S$ to the bottom-right corner of $S$, always moving directly right or directly down.
Each point where the path turns from moving down to moving right is the bottom-left corner of a rectangle of $R'$ that is not contained in $T$.
We index these rectangles $U_1,\ldots,U_m$ from top-left to bottom-right.
The left edge of $U_1$ is necessarily contained in $T$, or else we were wrong to index it as $U_1$.
Thus if $U_1$ fails to have both its bottom and left edges in $T$, then its bottom edge is not contained in $T$.
This implies that the left edge of $U_2$ is contained in $T$.
We continue until we find the first~$j$ such that the bottom edge of $U_j$ is contained in $T$.
Since the bottom edge of $U_m$ is in $T$, such a $j$ exists.
Necessarily, the left edge of $U_j$ is also contained in $T$.

We now consider the walls containing the edges of $U_j$.
First, let $W_l$ be the wall containing the left edge of $U_j$.
(If $j=1$ and the left edge of $U_j$ is in the left edge of $S$, then there is no wall shuffle associated to the left edge of $U_j$.)
Because the bottom edge of $U_{j-1}$ is not contained in $T$, the top endpoint of $W_l$ is contained in the bottom edge of $U_{j-1}$, as illustrated in Figure~\ref{surj fig}.
\begin{figure}
\scalebox{1}{\includegraphics{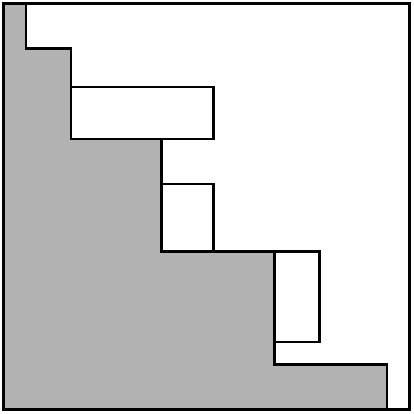}}
\begin{picture}(0,0)(123,0)
\put(49.5,54){$U_j$}
\put(32,84.5){$U_{j-1}$}
\put(82.5,30.5){$V$}
\put(25,25){$T$}
\end{picture}
\caption{A figure illustrating the proof of Proposition~\ref{gamma surj}}
\label{surj fig}
\end{figure}
(If $j=1$ and the left edge of $U_j$ is not in the left edge of $S$, then the top vertex of $W_l$ is in the top edge of $S$.)
We conclude that all rectangles adjacent to and left of $W_l$ are contained in $T$.
Thus we can pick $U_j$ in step $i$ consistent with the wall shuffle in $W_l$.

Second, let $W_t$ be the wall containing the top edge of $U_j$.
If $W_t$ is also the wall containing the bottom edge of $U_{j-1}$, then since the bottom edge of $U_{j-1}$ is not contained in $T$, the top-left corner $p$ of $U_j$ is the top-right corner of another rectangle of $R'$.
Since that other rectangle must intersect the diagonal, every point on $W_t$ from~$p$ rightwards is above the diagonal.
Thus $U_{j-1}$ is the rightmost of the rectangles adjacent to and above $W_t$, because otherwise the bottom-right corner of $U_{j-1}$ is the bottom-left corner of a rectangle of $R'$ that doesn't intersect the diagonal.
Since the left edge of $U_{j-1}$ is contained in $T$, all other rectangles adjacent to and above $W_t$ are constructed in steps $1$ through $i-1$.
Thus we can pick $U_j$ in step $i$ consistent with the wall shuffle in $W_t$.
If $W_t$ is not the wall containing the bottom edge of $U_{j-1}$, then the left endpoint of $W_t$ is also the top-left corner of $U_j$.
(This is the case that is illustrated in Figure~\ref{surj fig}.)
In this case, $U_j$ is leftmost among rectangles adjacent to and below $W_t$, so $U_j$ does not figure in the wall shuffle in~$W_t$.

We have shown that picking $U_j$ in step $i$ is allowed by Proposition~\ref{rho fibers} and by the wall shuffles in the walls $W_l$ and $W_t$.
Let $W_r$ be the wall containing the right edge of $U_j$ and let $W_b$ be the wall containing the bottom edge of $U_j$, if these exist.
If $j=m$, then there is no wall $W_r$ and either there is no wall $W_b$ or $U_j$ is the rightmost rectangle adjacent to and above $W_b$, so that $U_j$ does not figure in the wall shuffle in $W_b$.
Thus if $j=m$, the rectangle $U_j$ can be picked in step $i$.
If, on the other hand, $j<m$, then $U_j$ can be picked in step $i$ if and only picking it is allowed by the wall shuffle in $W_r$ and by the wall shuffle in $W_b$.

Let $W'_l$ be the wall containing the left edge of $U_{j+1}$ and let $W'_t$ be the wall containing the top edge of $U_{j+1}$.
We will prove the following claim:
If picking $U_j$ in step $i$ is disallowed by the wall shuffle in $W_r$, or if it is disallowed by the wall shuffle in $W_b$, then the left edge of $U_{j+1}$ is contained in $T$, and picking $U_{j+1}$ in step $i$ is allowed by the wall shuffle in $W'_l$ and by the wall shuffle in $W'_t$.

First, suppose that picking $U_j$ in step $i$ is disallowed by the wall shuffle in $W_r$.
If the bottom endpoint of $W_r$ is also the bottom-right corner of $U_j$ (as shown in Figure~\ref{surj fig}), then $U_j$ is the lowest of the rectangles adjacent to and left of $W_r$.
This would contradict the supposition that picking $U_j$ next is disallowed by the wall shuffle in $W_r$, so we conclude that the wall $W_r$ continues below $U_j$.
Since the bottom edge of $U_j$ is contained in $T$, it follows that the bottom right corner of $U_j$ is the next convex corner of $T$.
In particular, the wall shuffle in $W_r$ requires us to choose $U_{j+1}$ before $U_j$.
Since the wall $W'_l$ coincides with $W_r$, we know that the wall shuffle in $W'_l$ does not prevent us from choosing $U_{j+1}$ next.
Also, we see that the left edge of $U_{j+1}$ is contained in $T$:
Otherwise the right edge of $U_j$ intersects the left edge of $U_{j+1}$, making $U_{j+1}$ the topmost of the rectangles adjacent to and right of $W_r$ (because $R'$ is a diagonal rectangulation).
This contradicts the supposition that picking $U_j$ next is disallowed by the wall shuffle in $W_r$.
Furthermore, we see that the top-left corner of $U_{j+1}$ is strictly below the convex corner of $T$ separating $U_j$ from $U_{j+1}$:
Otherwise, that convex corner is the corner of four rectangles of $R'$ (including $U_j$ and $U_{j+1}$).
Thus $U_{j+1}$ is the leftmost rectangle adjacent to and below $W'_t$, so it does not figure in the wall shuffle in $W'_t$.

Next, suppose that picking $U_j$ in step $i$ is disallowed by the wall shuffle in $W_b$.
Let $V$ be leftmost among rectangles adjacent to and below $W_b$ that are not contained in $T$.
The rectangle $V$ is also shown in Figure~\ref{surj fig}.
Since picking $U_j$ in step $i$ is disallowed by the wall shuffle in $W_b$, the rectangle $V$ exists and comes before $U_{j}$ in the wall shuffle for $W_b$.
The top endpoint of the wall $W'_l$ is the top-left corner of $V$, and in particular is contained in $T$.
Thus all of the rectangles adjacent to and left of $W'_l$ are in $T$, so that choosing $U_{j+1}$ in step $i$ is allowed by the wall shuffle in $W'_l$.
If $V=U_{j+1}$, then $W'_t=W_b$, and we already know that the wall shuffle in $W_b$ requires us to pick $V$ next.  
If $V$ is not $U_{j+1}$, then the left endpoint of $W'_t$ is contained in $T$, so $U_{j+1}$ is the leftmost rectangle adjacent to and below $W'_t$, and thus $U_{j+1}$ does not figure in the wall shuffle in $W'_t$.
In either case, the left edge of $U_{j+1}$ is contained in $T$, and we have proved the claim.

If the bottom edge of $U_{j+1}$ is contained in $T$, then the claim implies that picking $U_{j+1}$ in step $i$ is allowed by Proposition~\ref{rho fibers} and by the wall shuffles in the walls $W'_l$ and $W'_t$.
We can thus argue for $U_{j+1}$ just as we have argued above for $U_j$.
If the bottom edge of $U_{j+1}$ is not contained in $T$, then we find the first $k>j$ such that the bottom edge of $U_k$ is contained in $T$, and start over as above, replacing $U_j$ with $U_k$.
Eventually, we will find a rectangle that can be picked, because the bottom edge of $U_m$ is contained in $T$ and because, as mentioned above, wall shuffles in the walls below $U_m$ and to the right of $U_m$ will never prevent its being picked.
\end{proof}

\begin{prop}\label{gamma fibers}
Suppose $x\covered y$ in the weak order, and let $e$ and $a$ be the adjacent entries that are swapped to convert $y$ to $x$, with $a<e$.
Then $\gamma(x)=\gamma(y)$ if and only if there are entries $b$, $c$, and $d$ in $y$, with $a<b<c<d<e$, such that $b$ and $d$ are on the same side of $ea$, while $c$ is on the other side of $ea$.
\end{prop}
\begin{proof}
Both conditions in the proposition imply that $\rho(x)=\rho(y)$, by the definition of $\gamma$ and by Proposition~\ref{rho fibers}.
Throughout the proof, let $R'$ be the diagonal rectangulation $\rho(x)=\rho(y)$.
We claim that $\gamma(x)=\gamma(y)$ if and only if $U_a$ and $U_e$ are not adjacent to any common wall of $R'$.
Indeed, if $U_a$ and $U_e$ are not adjacent to any common wall of $R'$, then $x$ and $y$ must define the same wall shuffles on the walls of $R'$, so $\gamma(x)=\gamma(y)$.
Conversely, suppose $U_a$ and $U_e$ are adjacent to a common wall $W$ of $R'$.
The assumption that $\rho(x)=\rho(y)$ rules out the possibility that $U_a$ and $U_e$ are both on the same side of $W$, so $U_a$ and $U_e$ are on opposite sides of $W$.
Since $U_a$ is chosen immediately before $U_e$ when $R'$ is constructed as $\rho(x)$ but immediately after $U_e$ when constructing $R'$ as $\rho(y)$, we see that $a$ and $e$ are both entries in~$\sigma_W$.
That is, if $W$ is vertical, then neither $U_a$ nor $U_e$ is the bottom-most rectangle adjacent to $W$ on the left, and neither $U_a$ nor $U_e$ is the topmost rectangle adjacent to $W$ on the right.
Similarly, if $W$ is horizontal, neither of the two rectangles are the leftmost rectangle below $W$ nor the rightmost rectangle above $W$.
We conclude that $\gamma(x)\neq\gamma(y)$ and we have proved the claim.

Suppose there are entries $b$, $c$, and $d$ in $y$, with $a<b<c<d<e$, such that $b$ and $d$ precede $ea$, but $c$ follows $ea$.
Let $T$ be the union of the left and bottom edges of $S$ with the rectangles chosen before $U_a$ and $U_e$ when $R'$ is constructed as $\rho(x)$ or $\rho(y)$.
In the construction of $R'$ as $\rho(x)$, the rectangle $U_a$ is chosen next, but in the construction of $R'$ as $\rho(y)$, the rectangle $U_e$ is chosen next.
Thus both $T\cup U_a$ and $T\cup U_e$ are bottom- and left-justified sets.
Therefore, every point of $U_e$ is strictly below and strictly to the right of every point of $U_a$.
See Figure~\ref{fiber fig}.a, ignoring, for now, the labels $p_a$ and $p_e$.
\begin{figure}
\begin{tabular}{cccc}
\includegraphics{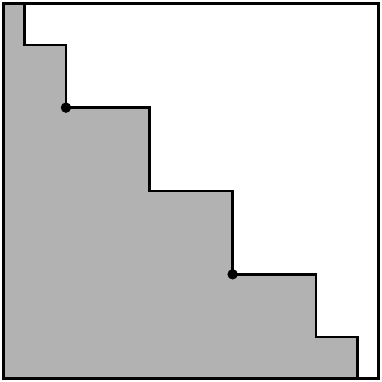}
\begin{picture}(0,0)(114,0)
\put(22,22){$T$}
\put(12,73){$p_a$}
\put(60,25){$p_e$}
\put(24,83){$a$}
\put(36,70){$b$}
\put(48,59){$c$}
\put(60,46){$d$}
\put(72,35){$e$}
\end{picture}
&&&
\includegraphics{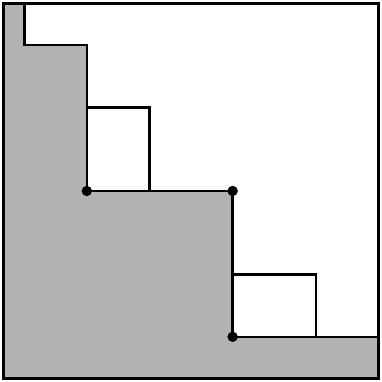}
\begin{picture}(0,0)(114,0)
\put(22,17){$T$}
\put(18,49){$p_a$}
\put(28,59){$U_a$}
\put(60,7){$p_e$}
\put(70,16.5){$U_e$}
\put(70,56){$p$}
\put(36,71.5){$a$}
\put(48,59){$b$}
\put(60,47){$c$}
\put(71,35){$d$}
\put(84,23.5){$e$}
\end{picture}
\\
(a)&&&(b)
\end{tabular}
\caption{Figures illustrating the proof of Proposition~\ref{gamma fibers}}
\label{fiber fig}
\end{figure}
In particular, $U_a$ and $U_e$ are not adjacent to a common wall of $R'$.
By the claim, $\gamma(x)=\gamma(y)$.

Similarly, if there are entries $b$, $c$, and $d$ in $y$, with $a<b<c<d<e$, such that $c$ precedes $ea$, but $b$ and $d$ follow $ea$, we see that every point of $U_e$ is strictly below and strictly to the right of every point of $U_a$. 
See Figure~\ref{fiber fig}.b, ignoring the labels $p$, $p_a$, and $p_e$.
In particular, $U_a$ and $U_e$ are not adjacent to a common wall of $R'$, so $\gamma(x)=\gamma(y)$.

Conversely, suppose $\gamma(x)=\gamma(y)$.
Let $T$ be as above.
For every concave corner $p$ of the top-right boundary of $T$, there is a rectangle of $U$ whose bottom-left corner is $p$.
Let $p_a$ be the bottom-left corner of $U_a$ and let $p_e$ be the bottom-left corner of $U_e$.
Both $p_a$ and $p_e$ are concave corners of the boundary of $T$.
There are two possibilities:
The first is that, looking from top-left to bottom-right at the concave corners of the boundary of $T$, there is some concave corner between $p_a$ and $p_e$.
In this case, there exist entries $b$, $c$, and $d$ in $y$, with $a<b<c<d<e$, such that $b$ and $d$ are before $ea$, while $c$ is after $ea$, as illustrated by Figure~\ref{fiber fig}.a.
The second possibility is that, from top-left to bottom-right, there are no other concave corners between $p_a$ and $p_e$.
Thus there is a single convex corner $p$ of $T$ between $p_a$ and $p_e$.
Since $R'$ equals both $\rho(x)$ and $\rho(y)$, both $U_a$ and $U_e$ have their bottom and left edges contained in $T$.
If $p$ is the bottom-right corner of $U_a$, or if $p$ is the top-right corner of $U_e$, then $U_a$ and $U_e$ are adjacent to a common wall.  
By the claim, this is a contradiction to the supposition that $\gamma(x)=\gamma(y)$.
Thus there exist entries $b$, $c$, and $d$ in $y$, with $a<b<c<d<e$, such that $c$ is before $ea$, while $b$ and $d$ are after $ea$, as illustrated in Figure~\ref{fiber fig}.b.
\end{proof}

Proposition~\ref{gamma surj} asserts that every fiber of $\gamma$ is nonempty.  
The following proposition characterizes the fibers more exactly, and completes the proof of Theorem~\ref{main}.
\begin{prop}\label{gamma bottoms}
Each fiber of $\gamma$ is a $\Gamma$-class.
In particular, each fiber of $\gamma$ contains a unique $2$-clumped permutation.
\end{prop}
\begin{proof}
By Proposition~\ref{Theta g props}.\ref{Theta g moves} and Proposition~\ref{gamma fibers}, the fibers of $\gamma$ are unions of $\Gamma$-classes.
Suppose $x$ and $y$ are distinct permutations with $\gamma(x)=\gamma(y)$.
Then $\rho(x)=\rho(y)$ and $x$ and $y$ are consistent with the same set of wall shuffles.
We will show that $x$ and $y$ are congruent modulo $\Gamma$.
Let $i$ be the smallest index such that $x_i\neq y_i$.
We argue by induction on $n-i$.

There is some $k>i$ such that $y_k=x_i$.
Since $\gamma(x)=\gamma(y)$, either the rectangle numbered $x_i$ or the sequence of rectangles numbered $y_iy_{i+1}\cdots y_k$ can be chosen next, consistent with the requirements of Proposition~\ref{rho fibers} and with the wall shuffles.
We conclude that the entry $y_k=x_i$ does not participate in any wall shuffles with any of the entries $y_iy_{i+1}\cdots y_{k-1}$.

Consider the sequence of permutations starting with $y$ and moving the entry $y_k$ to the left one place at a time, without changing the relative positions of the other entries, with the final entry $y'$ in the sequence having $y_k$ in position $i$.
Then~$\gamma$ is constant on the sequence.
Since each pair of adjacent permutations in the sequence is a covering pair in the weak order, each pair is related as described in Proposition~\ref{gamma fibers}.
But then Proposition~\ref{Theta g props}.\ref{Theta g moves} says that the entire sequence is contained in one $\Gamma$-class, so that in particular $y'$ and $y$ are congruent modulo $\Gamma$.
Since $\gamma(x)=\gamma(y')$ and $x$ and $y'$ agree in positions $1$ through $i$, by induction we conclude that $x$ and $y'$ are congruent modulo $\Gamma$.
Thus $x$ and $y$ are congruent modulo~$\Gamma$.

We have shown that each fiber of $\gamma$ is a $\Gamma$-class.
The second assertion of the proposition follows by Proposition~\ref{Theta g props}.\ref{Theta g bottoms}.
\end{proof}

\section{Remarks on enumeration}
A pleasant formula was obtained in~\cite{CGHK} for the number of Baxter permutations in $S_n$:
\[B(n)=\binom{n+1}{1}^{-1}\binom{n+1}{2}^{-1}\,\,\sum_{k=1}^n\binom{n+1}{k-1}\binom{n+1}{k}\binom{n+1}{k+1}.\]
This formula applies to twisted Baxter (i.e.\ $1$-clumped) permutations and to diagonal rectangulations as well.
In this section, we make several remarks on the problem of enumerating generic rectangulations or $2$-clumped permutations.
In particular, we give some indications that the enumeration of $2$-clumped permutations will be harder than the enumeration of $1$-clumped permutations.

\begin{remark}\label{Conant and Michaels}
One way to enumerate generic rectangulations is by specializing a formula of Conant and Michaels~\cite{ConMic}.
This formula is a recursion, with signs, counting rectangulations according to the number of crosses.
Thanks to Jim Conant for providing the results of his recursive calculations which verify and extend Table~\ref{comput table}.
\end{remark}

\begin{remark}\label{my computations}
Another approach to enumerating $2$-clumped permutations is to apply the key idea from~\cite{CGHK}.
This approach appears not to lead to a formula for the number of generic rectangulations, but is useful computationally, as we now explain.

Suppose $x\in G_n$.
For each entry $a$ in $x$, let $\beta(a)=\set{b\in [a+1,n]: b\mbox{ is before }a}$.
Then $n+1$ can be placed before $a$ in $x$ to obtain another $2$-clumped permutation if and only if one of the following holds:
$\beta(a)=\emptyset$, $\beta(a)=[a+1,n]$, $\beta(a)=[a+1,c]$ for some $c$ with $a+1\le c<n$, or $\beta(a)=[d,n]$ for some $d$ with $a+1<d\le n$.
Notice that if $a$ satisfies none of these requirements, then even after $n+1$ is inserted elsewhere to obtain a permutation $x'\in G_{n+1}$, the entry $a$ in $x'$ still satisfies none of the requirements.
Notice also that $n+1$ can be inserted after all of the entries of $x$ to obtain a permutation in $G_{n+1}$.

Accordingly, we encode a $2$-clumped permutation by a string of letters as follows.
Read through the elements of $x$ from left to right, and for each element $a$, write a letter in the string as follows:

\begin{tabular}{lll}
\textbf{n}&(for ``null'' or ``$n$'')& if $a=n$.  Assume $a\neq n$ in the following cases.\\
\textbf{e}& (for ``empty'') &if $\beta(a)=\emptyset$.\\
\textbf{f}& (for ``full'') &if $\beta(a)=[a+1,n]$.\\
\textbf{l}& (for ``lower'') &if $\beta(a)=[a+1,c]$ for some $c$ with $a+1\le c<n$.\\
\textbf{u}&(for ``upper'') &if $\beta(a)=[d,n]$ for some $d$ with $a+1<d\le n$.
\end{tabular}\\
If none of these apply, then write nothing.


For example, for each of the permutations $2413$, $4213$, $3124$ and $3142$, the symbol~$5$ can be inserted anywhere except before the symbol $1$.
The sequences of letters for these permutations are respectively \textbf{en}$\,\cdot\,$\textbf{f}, \textbf{nu}$\,\cdot\,$\textbf{f}, \textbf{e}$\,\cdot\,$\textbf{ln}, and \textbf{e}$\,\cdot\,$\textbf{nf}, with a dot ``$\cdot$'' indicating an entry in the permutation that does not produce a letter.
The respective strings are \textbf{enf}, \textbf{nuf}, \textbf{eln}, and again \textbf{enf}.

If we place the symbol $n+1$ before $a$ in $x$ or if we place $n+1$ after all entries of $x$, we can construct the string of letters corresponding to the new permutation $x'$ by the following procedure.
Insert the letter \textbf{n} in the string before the letter corresponding to $a$ or at the end of the string and alter letters before the insertion according to the following rule: \textbf{n} becomes \textbf{e}, \textbf{e} is unchanged, \textbf{f} becomes \textbf{l}, \textbf{l} is unchanged, and \textbf{u} is deleted.
Alter letters occurring after the insertion as follows: \textbf{n} becomes \textbf{f}, \textbf{e} becomes \textbf{u}, \textbf{f} is unchanged,  \textbf{l} is deleted, and \textbf{u} is unchanged.

Now we can dispense with permutations entirely and simply insert letters into strings, counting the resulting strings by multiplicities.
We start with the string \textbf{n}, encoding the permutation $1\in G_1$.
Inserting before or after the one letter in the string, we obtain \textbf{en} and \textbf{nf}, corresponding to the permutations $G_2=\set{12,21}$.
Inserting into these two strings, we obtain the strings \textbf{een} (for $123$), \textbf{enf} twice (for $132$ and $231$), \textbf{nuf} (for $312$), \textbf{eln} (for $213$), and \textbf{nff} (for $321$).
In the next round of insertions, deletions of letters come into play, so that for example, inserting \textbf{n} after the \textbf{e} in  \textbf{eln}, we obtain \textbf{enf}.
This corresponds to inserting $4$ after $2$ in $213$ to obtain $2413$.
In all, there are 15 strings which represent the 24 permutations in $G_4$.

The values shown in Table~\ref{comput table} are the results of a simple computer program that generates all strings and keeps track of multiplicities.


In contrast, representing $1$-clumped permutations (i.e.\ twisted Baxter permutations) by strings leads to an enumeration formula.
In this case the locations where $n+1$ can be inserted are the locations labeled \textbf{n}, \textbf{e}, or \textbf{f}, with the same definitions as above.  
When \textbf{n} is inserted into the string, the remainder of the string is altered as follows:
Before the insertion, \textbf{n} becomes \textbf{e}, \textbf{e} is unchanged, and \textbf{f} is deleted.
After the insertion, \textbf{n} becomes \textbf{f}, \textbf{e} is deleted, and \textbf{f} is unchanged.
All of the strings are of the form $\textbf{e}^i\textbf{n}\textbf{f}\,^j$, for $i,j\ge 0$ and $i+j\le n-1$.
Define $G(n,i,j)$ to be the multiplicity of the string $\textbf{e}^i\textbf{n}\textbf{f}\,^j$ for $1$-clumped permutations in $S_n$.
Up to reindexing in $n$, the numbers $G(n,i,j)$ coincide with the numbers $T_n(i,j)$ in \cite{CGHK}, and the obvious recurrence on $G(n,i,j)$ coincides with the recurrence on $T_n(i,j)$.
This recurrence can be solved as in \cite{CGHK}, or by the generating function method of~\cite{B-Mel}.
In particular, the generating tree for the twisted Baxter permutations is isomorphic to the generating tree for Baxter permutations.
Indeed, the original proof \cite{West pers} that twisted Baxter permutations biject with Baxter permutations proceeded by establishing this isomorphism of generating trees.
\end{remark}

\begin{remark}\label{formula difficulties}
Mallows \cite{Mallows} gave a combinatorial interpretation for the terms in formula for $B(n)$ by pointing out that the term indexed by $k$ counts Baxter permutations with $k$ \emph{ascents} (or \emph{rises}).
There are two dual ways to define ascents:
We will say that a \emph{right ascent} is a pair of adjacent entries such that the left entry in the pair is smaller than the right entry in the pair.
A \emph{left ascent} is a pair of entries $i$  and $i-1$ with $i-1$ appearing before $i$ in the permutation.
We can similarly define \emph{right descents} (left entry in the pair larger) and \emph{left descents} ($i-1$ appearing after $i$).
Recall that the Baxter permutation are the permutations avoiding $3$-$14$-$2$ and $2$-$41$-$3$.
It is easy to see that a given permutation is a Baxter permutation if and only if its inverse is a Baxter permutation.
(See e.g.\ \cite[Corollary~4.2]{rectangle}.)
Thus, when counting Baxter permutations according to the number of ascents, it does not matter whether we use right ascents or left ascents.
Furthermore, it is immediate that a permutation is a Baxter permutation if and only if its reverse permutation is also a Baxter permutation.
Thus the formula for Baxter permutations with a fixed number of ascents is the same as the formula for Baxter permutations with a fixed number of descents.
It is easy to see that the number of ascents in a permutation~$x$ equals the number of vertical walls in the diagonal rectangulation $\rho(x)$.
Thus the formula for $B(n)$ counts diagonal rectangulations according to the number of vertical walls.

The number of left ascents of $x$ also equals the number of vertical walls in the generic rectangulation $\gamma(x)$ and the number of left descents of $x$ equals the number of horizontal walls.
Thus by the symmetry of the rectangulations, counting $2$-clumped permutations by left descents is equivalent to counting $2$-clumped permutations by left ascents.
However, the inverse of a $2$-clumped permutation is not necessarily a $2$-clumped permutation, so it matters whether we take the left or right definitions of descents or ascents.
Thus there are at least three reasonable statistics by which to count:  left ascents/descents, right ascents, or right descents.
Computations show that these three statistics are distributed differently, and suggest the following conjecture:
\begin{conj}\label{right desc conj}
Fix $k\ge 0$.
Then for $n\ge 1$, the number of $2$-clumped permutations in $S_n$ with exactly $d$ right descents is a polynomial $p_k(n)$ of degree $3d$ and leading coefficient 
\[\prod_{i=1}^d\frac{2}{i(i+1)(i+2)}=\frac{2^{d+1}}{d!(d+1)!(d+2)!}.\]
\end{conj}
The polynomial $p_k(n)$ must have factors $(n-1)(n-2)\cdots(n-d)$, so the point is to determine the polynomial $\tilde{p}_k(n)$ of degree $2d$ that results when these factors are taken out.
The first few polynomials appear to be $\tilde{p}_0(n)=1$,
\begin{align*}
\tilde{p}_1(n)&=(n^2-2n+3)/3,\\
\tilde{p}_2(n)&=(5n^4-36n^3+142n^2-279n+270)/180,\mbox{ and}\\
\tilde{p}_3(n)&=(14n^6-213n^5+1688n^4-8361n^3+26000n^2-46884n+37800)/15120.
\end{align*}
It should be emphasized that the point of the conjecture is to find a formula enumerating all $2$-clumped permutations.
The conjecture can be proved for some small values of $k$, and proofs for additional values of $k$ are of interest only to the extent to which they lead to a conjecture on the general form of $p_k(n)$.

The other two statistics (left ascents/descents and right ascents) do not lead to polynomial formulas.
In particular, counting $2$-clumped permutations by left ascents, or equivalently counting generic rectangulations by the number of vertical walls, appears to be hard.
\end{remark}

\end{document}